\documentclass[12pt]{article}
\usepackage[english]{babel}
\usepackage{amsfonts,amsmath,amsxtra,amsthm,amssymb,latexsym}
\usepackage{color,soul}

\newcommand{\N}{{\mathbb N}}
\newcommand{\R}{{\mathbb R}}
\newcommand{\C}{{\mathbb C}}
\newcommand{\Z}{{\mathbb Z}}

\newcommand{\kS}{{\mathcal S}}

\newcommand{\kK}{{\mathcal K}}

\newcommand{\supp}{\mathop{\rm supp}\nolimits}

\newtheorem{theorem}{Theorem}[section]

 \newtheorem{corollary}[theorem]{Corollary}
 \newtheorem{lemma}[theorem]{Lemma}
 \newtheorem{proposition}[theorem]{Proposition}
 \theoremstyle{definition}
 \newtheorem{definition}[theorem]{Definition}
 \theoremstyle{remark}
 \newtheorem{remark}[theorem]{Remark}
 \newtheorem{example}[theorem]{Example}
 \numberwithin{equation}{section}

\DeclareMathOperator{\loc}{loc}

\author{L. Golinskii, M.~Malamud,  and L. Oridoroga }

\begin{title}
{Radial positive definite functions and Schoenberg matrices with negative eigenvalues}
\end{title}
\date{}
\begin{document}
\maketitle

\begin{abstract}
The main object under consideration is a class $\Phi_n\backslash\Phi_{n+1}$ of radial
positive definite functions on $\R^n$ which do not admit \emph{radial positive definite
continuation} on $\R^{n+1}$. We find certain necessary and sufficient conditions for the 
Schoenberg representation measure $\nu_n$ of  $f\in \Phi_n$
in order that the inclusion $f\in \Phi_{n+k}$, $k\in\N$, holds.
We show that the class $\Phi_n\backslash\Phi_{n+k}$ is rich enough by giving a number of
examples. In particular, we give a direct proof of $\Omega_n\in\Phi_n\backslash\Phi_{n+1}$, 
which avoids Schoenberg's theorem, $\Omega_n$ is the Schoenberg kernel.  We show that
$\Omega_n(a\cdot)\Omega_n(b\cdot)\in\Phi_n\backslash\Phi_{n+1}$, for $a\not=b$.
Moreover, for the square of this function we prove surprisingly much stronger result:
$\Omega_n^2(a\cdot)\in\Phi_{2n-1}\backslash\Phi_{2n}$.
We also show that any $f\in\Phi_n\backslash\Phi_{n+1}$, $n\ge2$, has infinitely many
negative squares. The latter means that for an arbitrary positive integer $N$ there is a
finite Schoenberg matrix $\kS_X(f) := \|f(|x_i-x_j|_{n+1})\|_{i,j=1}^{m}$,  
$X := \{x_j\}_{j=1}^m \subset\R^{n+1}$, which has at least $N$ negative eigenvalues.
\end{abstract}


\textbf{Key  words}. Positive definite functions, Schoenberg kernels,
hypergeometric functions, Hankel transform, Bessel functions

%%%\end{frontmatter}

%\renewcommand{\contentsname}{Contents}
%\tableofcontents

%%%%%%%%%%%%%%%%%%%%%%%%%%%%%%%
%%

\section{Introduction}

Positive definite functions have a long history, being an important chapter in various areas of
harmonic analysis. They can be traced back to papers of Carath\'eodory, Herglotz, Bernstein, culminating
in Bochner's celebrated theorem from 1932--1933.

In this paper we will be dealing primarily with radial positive definite functions. Such functions
have significant applications in probability theory, statistics, and approximation theory,
where they occur as the characteristic functions or Fourier transforms of spherically symmetric probability
distributions. Denote the class of radial positive definite functions on $\R^n$ by $\Phi_n$.

We follow the standard notation for the inner product $(u,v)_n=(u,v)=u_1v_1+\ldots+u_nv_n$ of two vectors
$u=(u_1,\ldots,u_n)$ and $v=(v_1,\ldots,v_n)$ in $\R^n$, and $ |u|_n=|u|=\sqrt{(u,u)} $ for the Euclidean norm of~$u$.

%%%%%%%%%%%%%%%%%%%%%%%%%%%%%%%%%%%
   \begin{definition} Let $n\in \N$.
A real-valued and continuous function $f$ on $\R_+=[0,\infty)$, $f\in C(\R_+)$, is
called a {\it radial positive definite (RPD) function on $\R^n$}, if for an arbitrary
finite set $\{x_1,\dots,x_m\}$, $x_k\in \R^n$, and $\{\xi_1,\dots,\xi_m\}\in\C^m$
\begin{equation}\label{positiv}
\sum_{k,j=1}^{m} f(|x_k-x_j|_n)\xi_j\overline{\xi}_k\ge 0.
\end{equation}
  \end{definition}
In other words, RPD functions $f$ are exactly those, for which
$f(|\cdot|_n)$ are positive definite functions on $\R^n$ or, equivalently, the Schoenberg
matrices $\kS_X(f) := \|f(|x_i-x_j|_{n})\|_{i,j=1}^{m}$  are  positive definite for any
 $X := \{x_j\}_{j=1}^m \subset\R^{n}$.

The characterization of radial positive definite functions is a fundamental result of I. Schoenberg
\cite{Sch38, Sch38_1} (see, e.g., \cite[Theorem~5.4.2]{Akh65}).

%%%%%%%%%%%%%%%%%%%%%%%%%%%%%%%%%
  \begin{theorem}\label{schoenbergtheorem}
 A function  $f\in\Phi_n$, $f(0)=1$, if and only if there exists
 a probability measure $\nu$ on $\R_+$ such that
  \begin{equation}\label{schoenberg1}
f(r) = \int_{0}^{\infty}\Omega_n(rt)\,\nu(dt), \qquad r\in \R_+,
   \end{equation}
where the Schoenberg kernel
  \begin{equation}\label{kernel}
  \Omega_n(s):=\Gamma(q+1)\,
  \left(\frac{2}{s}\right)^q\,J_q(s)=
  \sum_{j=0}^{\infty}\frac{\Gamma(q+1)}{j!\,\Gamma(j+q+1)}\,
  \left(-\frac{s^2}{4}\right)^j, \quad q:=\frac{n}2-1,
 \end{equation}
$J_q$ is the Bessel function of the first kind and order $q$. Moreover,
   \begin{equation}\label{fouriersphere}
\Omega_n(|x|) = \int_{S^{n-1}}e^{i(u,x)}\sigma_n(du), \qquad
x\in\R^n,
  \end{equation}
where $\sigma_n$ is the normalized surface measure on the unit sphere $S^{n-1}\subset\R^n$.
  \end{theorem}
So it is not surprising that various properties of the Bessel functions (recurrence and differential relations, bounds and asymptotics, integrals) come up repeatedly throughout the paper. The first three functions $\Omega_n$, $n=1,2,3,$ can be computed as
 \begin{equation}\label{3.13}
   \Omega_1(s)=\cos s,\quad  \Omega_2(s)=J_0(s),\quad  \Omega_3(s)=\frac{{\rm sin} s}{s}\,.
 \end{equation}

It is well known that the classes $\Phi_n$ are nested, and inclusion $\Phi_{n+1}\subset\Phi_n$ is proper for any $n\in \N$. The result is mentioned in the pioneering paper of Schoenberg \cite{Sch38} (without proof), and then duplicated in Akhiezer's book \cite{Akh65} (and in a number of later papers and books). The main goal of
Section 2 is to study the classes $\Phi_n\backslash\Phi_{n+1}$. We start out with two proofs of the known fact $\Omega_n\in\Phi_n\backslash\Phi_{n+1}$, the first one is based on the Schoenberg theorem and some rudiments
of the Stieltjes moment problem. The second one is direct and has nothing to do with Schoenberg's
theorem. We just manufacture a set $X=\{x_j\}_{j=1}^{n+3}\subset\R^{n+1}$ such that the corresponding matrix
$\kS_X(\Omega_n)=\|\Omega_n(|x_i-x_j|)\|_{i,j=1}^{n+3}$ called the {\it Schoenberg matrix} (see \cite{gmo1} for a detailed account of this object) has at least one negative eigenvalue (Proposition \ref{negativeeigen}).

Let us emphasize that the inclusion $f\in \Phi_n\backslash\Phi_{n+1}$  means that $f$ is RPD
function on $\R^n$ but does not admit \emph{radial} positive definite continuation on
$\R^{n+1}$, while positive definite continuations obviously exist.

As it turns out, the class $\Phi_n\backslash\Phi_{n+1}$ is rich enough. We give some
sufficient conditions in terms of Schoenberg's measure $\nu_{n}=\nu_n(f)$ for $f$ to
belong to $\Phi_n\backslash\Phi_{n+1}$. A key ingredient here is the following relation
called a {\it transition formula}.

Let $f\in\Phi_{n}$, $n\ge2$. Then $f\in\Phi_m$ for $m=1,2,\ldots,n-1$, and according to
Schoenbeerg's theorem $f$ admits representation \eqref{schoenberg1} with some measure
$\nu_m(f)$. We show (see Theorem \ref{prop_measure_relation}) that the relation between
the measures $\nu_n$ and $\nu_m$ is given by
  \begin{equation}\label{transfor}
\nu_m(dx)=p_m(x)\,dx, \quad p_m(x) =
\frac{2x^{m-1}}{B\bigl(\frac{m}2,\frac{n-m}2\bigr)}\,\int_x^\infty
\Bigl(1-\frac{x^2}{u^2}\Bigl)^{\frac{n-m}2-1}\,\,\frac{\nu_n(du)}{u^m}\,, %B(a,b):=\frac{\Gamma(a)\Gamma(b)}{\Gamma(a+b)}
\end{equation}
where $B(a,b)$ is the Euler beta function. So each $\nu_m$ is absolutely continuous and
$\nu_m\{(0,\varepsilon)\}>0$ for any $\varepsilon>0$. In other words, if $\nu_n(f)$ is either
not pure absolutely continuous (contains a singular component) or
$\nu_n\{(0,\varepsilon)\}=0$ for some $\varepsilon>0$, then $f\notin\Phi_{n+1}$.

Besides, we investigate the smoothness and decaying properties of the distribution function generated
by the Schoenberg's measure $\nu_{m}$ in \eqref{transfor}. For instance, it is shown in
Theorem \ref{prop_smoothness}  that $\nu_m\in AC^{[k/2]}_{\loc}(\R_+)$ and moreover,
$x^j p^{(j)}_m(x)\in L^p(\Bbb R_+)$, $j=0,1,\ldots, [k/2]-1$, $k=n-m$, whenever $\nu_n$
is absolutely continuous, $\nu_n = p_{n}\,dx$, and $p_{n}\in L^p(\Bbb R_+)$
for some $1\le p<\infty$.

Clearly, $\Omega_n(a\cdot)\in\Phi_n$ for each $a>0$. Since the Schur (entrywise) product of two
nonnegative matrices is again a nonnegative matrix, the product of two Schoenberg's kernels $\Omega_n(a\cdot)\Omega_n(b\cdot)\in\Phi_n$, $a,b>0$. Moreover, the following
result is proved in Section 4.

\begin{theorem}\label{schoenkern}
For $n\in\N$
\begin{enumerate}
  \item [\em (i)] $\Omega_n(at)\Omega_n(bt)\in\Phi_n\backslash\Phi_{n+1}$, \ \ $a\not=b$;
  \item [\em (ii)] $\Omega_n^2(t)\in\Phi_{2n-1}\backslash\Phi_{2n}$.
  \end{enumerate}
\end{theorem}

The problem we deal with in Section \ref{secinf} concerns the number of negative squares
of  a function  $f\in\Phi_n\backslash\Phi_{n+1}$. Namely, since such $f$ does not admit
RPD continuation on $\R^{n+1}$,  the quadratic forms
\eqref{positiv} associated with  $f(|\cdot|_{n+1})$ in place of $f(|\cdot|_{n})$ might
have negative squares. We are interested in the maximal number of negative squares of
such  forms.
One reformulates this concept in terms of the maximal  number of negative  eigenvalues
of the corresponding Schoenberg's matrices $\kS_X(f) := \|f(|x_i-x_j|_{n+1})\|_{i,j=1}^{m}$
with $X= \{x_j\}_{j=1}^m \subset\R^{n+1}$. To  be precise, given a real-valued and continuous 
function $g$ on $\R_+$ and a finite set
$X\subset\R^n$, denote by $\kappa^-(g, X)$ a number of negative eigenvalues of the
finite Schoenberg matrix $\kS_X(g)$, and by
  $$
\kappa_n^-(g):= \sup \{\kappa^-(g,X): X\  \text{runs through all finite subsets of}\
\R^n\}.
  $$
Certainly, $\kappa_n^-(g)=0$ for $g\in\Phi_n$. The question is whether $\kappa_n^-(g)$
can be finite for $g\notin\Phi_n$.

\begin{theorem}\label{th1}
Let $g$ be a continuous function on $\R_+$ such that the limit
\begin{equation}\label{limit}
\lim_{t\to\infty} g(t)=g(\infty)\ge0
\end{equation}
exists. If $\kappa^-(g,Y)\ge1$ for some set $Y\in\R^m$, then $\kappa_m^-(g)=+\infty$. In particular,
  \begin{equation}\label{infinitesquares}
\kappa_{n+1}^-(g)=+\infty \quad {\rm for\ each}\quad g\in\Phi_n\backslash\Phi_{n+1}, \quad n\ge2.
\end{equation}
\end{theorem}

The case $n=1$ is more subtle, since $\Omega_1(t)=\cos t$ has no limit at infinity.
Nonetheless we believe that the conclusion in \eqref{infinitesquares} holds in this case as well.

{\bf{Conjecture.}} {\it For each function $f\in\Phi_1\setminus \Phi_2$ the relation
$\kappa_2^-(f)=+\infty$ holds.}

We confirm this conjecture for $f\in\Phi_1\backslash\Phi_2$ under certain additional
assumptions on the Schoenberg measure $\nu_1(f)$ (see \eqref{nontozero}).

In connection with  relation \eqref{infinitesquares} we note that 
functions with {\it finite} number of negative squares (indefinite analogs of positive definite functions) 
appear naturally in various extension problems. According to Theorem \ref{th1} this is not the case
for functions $f\in\Phi_n\backslash\Phi_{n+1}$.

We note also that indefinite analogs of positive definite and more general classes of functions have 
thoroughly been investigated by M. Krein and H. Langer (see \cite{KreLan14} and references therein).

\section{Functions from $\Phi_n\backslash\Phi_{n+1}$: algebraic approach}\label{schoennegativeeigen}

As we mentioned above the classes $\Phi_n$ are nested, and inclusion $\Phi_{n+1}\subset\Phi_n$ is proper for any $n\in \N$. For some examples of functions $f\in\Phi_n\backslash\Phi_{n+1}$ see, e.g., \cite[Remark 3.5]{GMZ11}, \cite{zas2000}.

There is a simple way to show that $\Omega_n\notin\Phi_{n+1}$, which relies on some basics from the
Stieltjes moment problem. We sketch the proof without getting into much details. Assume on the contrary that $\Omega_n\in\Phi_{n+1}$, and so
$$ \Omega_n(r)=\int_0^\infty \Omega_{n+1}(rt)\,\sigma(dt), \qquad \sigma(\R_+)=1, $$
As the function in the left hand side is an even entire function, it is easy to see that $\sigma$ has
all moments finite. By using the Taylor series expansions \eqref{kernel} for both $\Omega_n$ and $\Omega_{n+1}$
we come thereby to the following moment problem
$$ s_0=1, \quad s_{2k}:=\int_0^\infty t^{2k}\sigma(dt)=\frac{(n+1)(n+3)\ldots(n+2k-1)}{n(n+2)\ldots(n+2k-2)}\,, \qquad k=1,2,\ldots,  $$
in particular,
$$ s_2=\frac{n+1}{n}, \qquad s_4=\frac{(n+1)(n+3)}{n(n+2)}. $$
But such moment problem has no solution, since, e.g.,
$$ \det\begin{bmatrix}
   s_0& s_2\\
   s_2& s_4 \end{bmatrix}=\frac{(n+1)(n+3)}{n(n+2)}-\biggl(\frac{n+1}{n}\biggr)^2<0. $$

Our goal here is to suggest a direct proof of the relation $\Omega_n\notin\Phi_{n+1}$, which avoids Schoenberg's theorem. In other word, we construct an explicit finite set $X\subset\R^{n+1}$ so that the Schoenberg matrix $\kS_X(\Omega_n)$ has at least one negative eigenvalue, whereas $\kS_Y(\Omega_n)\ge0$ for each finite set $Y\subset\R^{n}$. Moreover, we show that there is no upper bound for the number of negative eigenvalues of $\kS_X(\Omega_n)$ for an appropriate choice of the set $X$.

Let $E_{l,m}$ be a $l\times m$ matrix composed of $1$'s, that is,
$(E_{l,m})_{ij}=1$, $1\le i\le l$, $1\le j\le m$, $E_m:=E_{m,m}$. It is clear that $rank\,E_m=1$ and its spectrum
$\sigma(E_m)=\{0^{(m-1)}, m\}$. Let $I_p$ be the unit matrix of order $p$.

\begin{lemma}\label{signature}
Let
\begin{equation}\label{linalg}
\kS:=\begin{bmatrix}
 A&B\\ B^*&I_k
    \end{bmatrix}, \end{equation}
where a $p\times p$ matrix $A$ and a $q\times k$ matrix $B$ are defined as
\begin{equation}\label{ABC}
%\begin{split}
A =\begin{bmatrix}
1&a&a&\ldots&a \\
a&1&a&\ldots&a \\
\vdots& & & & \vdots \\
a&a&a&\ldots& 1 \end{bmatrix}=(1-a)I_{p}+aE_{p}, \qquad
B=\begin{bmatrix}
b&b&\ldots&b \\
\vdots& & &\vdots \\
b&b&\ldots&b \end{bmatrix} =bE_{p,k}
%\end{split}
\end{equation}
with real entries $a$, $b$. $\kS$ has at least one negative eigenvalue if and only if either
of two inequalities holds
\begin{equation}\label{smallest}
a>1, \qquad \lambda:=1+(p-1)a-kpb^2<0.
\end{equation}
\end{lemma}
\begin{proof}
A general linear algebraic identity for the block matrix $\kS$
\begin{equation}\label{blockdiagonal}
\kS
= \begin{bmatrix}
 I_p&B\\ 0&I_k
    \end{bmatrix} \begin{bmatrix}
 A-BB^*&0\\ 0&I_k
    \end{bmatrix} \begin{bmatrix}
 I_p&0\\B^*&I_k
    \end{bmatrix}
\end{equation}
shows that $\kS$ has the same number of negative eigenvalues as the block diagonal matrix in the right
hand side of \eqref{blockdiagonal}, or the same number of negative eigenvalues as the matrix $A-BB^*$.

In our particular case the spectrum of the latter matrix can be computed explicitly. Indeed, as
$$ E^*_{l,m}=E_{m,l}, \qquad E_{m,l}E_{l,m}=lE_m, $$
then by \eqref{ABC}
\begin{equation*}
D=A-BB^*=(1-a)I_p+(a-kb^2)E_p, \qquad \sigma(D)=\{(1-a)^{(p-1)}, 1+(p-1)a-kpb^2\},
\end{equation*}
and the result follows.  \end{proof}
%%%%%%%%%%%%%%%%%%%%%%%%%%%%%%%%%%%%%%%%%%%%%%%%%%%%%%%%%

The matrix $\kS$ \eqref{linalg}--\eqref{ABC} (with $k=1$) will arise as the Schoenberg matrix
$\kS_X(f)$ for a certain configuration  $X$ in $\R^{n+1}$.
%%%%%%%%%%%%%%%%%%%%%%%%%%%%%%%%%%%%%%%%%%%%%%%%%%%%%%%%%%%%%%

\begin{proposition}\label{negativeeigen}
Let $f\in C(\R_+)$ be a real-valued function, $f\le1$, analytic at the origin with the Taylor series expansion
\begin{equation}\label{taylor}
f(z)=\sum_{j=0}^\infty (-1)^j a_j z^{2j}, \qquad a_0=1, \quad a_j>0, \quad j\in\N.
\end{equation}
Then $f\notin\Phi_{m+1}$ provided that
\begin{equation}\label{condtaylor}
\frac{2m+6}{m+1}\,a_2<a_1^2.
\end{equation}
In particular, $\Omega_n\in\Phi_n\backslash \Phi_{n+1}$.
\end{proposition}
%%%%%%%%%%%%%%%%%%%%%%%%%%%%%%%%%%%%%%%%%%%%%%%%%
\begin{proof}
Let  $X=\{x_j\}_{j=1}^{m+3}$ be a  configuration in $\R^{m+1}$ such
that the first $m+2$ points are the vertices of a regular simplex in
$R^{m+1}$ with the edge length $t$ and  $x_{m+3}$ be the center of this simplex.
Clearly,
\begin{equation}\label{simplex}
|x_i-x_j|=t, \quad i,j=1,2,\ldots,m+2, \quad
|x_{m+3}-x_i|=\rho_n t, \quad \rho_m := \sqrt{\frac{m+1}{2(m+2)}}\,,
\end{equation}
where $\rho_m$ is the radius of the circumscribed sphere. It is
easily seen  that $\kS_X(f)=\kS$, where $\kS$ is given by \eqref{linalg} with
$$
p=m+2, \quad k=1, \quad a=a(t)=f(t), \quad b=b(t)=f(\rho_m t),
$$
From \eqref{taylor} and \eqref{condtaylor} we see that for small enough $t$
$$
\lambda_m(t)=1+(m+1)f(t)-(m+2)f^2(\rho_m t)=\sum_{j=1}^\infty (-1)^j\lambda_{m,j}t^{2j}
$$
with
$$ \lambda_{m,1}=0, \quad \lambda_{m,2}=\frac{m+1}{4(m+2)}\bigl((2m+6)a_2-(m+1)a_1^2\bigr)<0. $$
So
$$ \lambda_m(t)=\lambda_{m,2}t^4+O(t^6)<0 $$
for small enough $t$, and the result follows from Lemma \ref{signature}.

The series expansion for $\Omega_n$ is
\begin{equation*}
\Omega_n(t) =1-\frac{t^2}{2n}+\frac{t^4}{8n(n+2)}+O(t^6), \qquad t\to 0,
\end{equation*}
so
\begin{equation*}
a_1^2-\frac{2n+6}{n+1}\,a_2=\frac1{4n}\Bigl(\frac1{n}-\frac{n+3}{(n+1)(n+2)}\Bigr)=\frac1{2n^2(n+1)(n+2)}>0.
\end{equation*}
The proof is complete.
\end{proof}

%%%%%%%%%%%%%%%%%%%%%%%%%%%%%%%%%%%%%%%%%%%%%%%%%%%%%%%%%%%%%%%%%%%%%%
\begin{remark}\label{wideclass}
Proposition \ref{negativeeigen} applies to a wide class of functions, e.g., to the powers of the
Schoenberg kernels $\Omega_n^p$, $p\in\N$, but the result obtained this way is far from being optimal
(at least for $p=2$). For instance, for the squares and cubes we find
$$ \frac{a_1^2(\Omega_n^2)}{a_2(\Omega_n^2)}=\frac{2n+4}{n+1}\,, \qquad \frac{a_1^2(\Omega_n^3)}{a_2(\Omega_n^3)}=\frac{6n+12}{3n+4}\,, $$
and \eqref{condtaylor} holds with $m\ge 2n+2$ and $m\ge 3n+4$, respectively. Hence,
$\Omega_n^2\notin\Phi_{2n+3}$ and $\Omega_n^3\notin\Phi_{3n+5}$
As a matter of fact we show later in Proposition \ref{square} that $\Omega_n^2\notin\Phi_{2n}$.
\end{remark}

\begin{remark}
If $X=\{x_k\}_{k=1}^{n+2}$ is the set of vertices of a regular simplex (without its center), then the
corresponding  Schoenberg matrix  $\kS_X(\Omega_n)$  is positive definite. Nonetheless, we conjecture that
a negative eigenvalue can already be seen on a certain configuration $Y=\{y_k\}_{k=1}^{n+2}$ of
$n+2$ points in $\R^{n+1}$.
\end{remark}

\section{When a function from the class $\Phi_n$ belongs to  $\Phi_{n+k}?$}

As a matter of fact, the class $\Phi_n\backslash\Phi_{n+1}$ is rather rich. We give some sufficient
conditions in terms of Schoenberg's measure $\nu_{n}=\nu_n(f)$ for $f$ to belong to
$\Phi_n\backslash\Phi_{n+1}$.

%%%%%%%%%%%%%%%%%%%%%%%%%%%%%%%%%%%%%%%%%%%%%%%%%%%%%%%
\begin{theorem}\label{prop_measure_relation}
Let  $(const\not=)f\in\Phi_{m}$, and let $\nu_{m}$ be its Schoenberg's
measure. Then $f\in\Phi_{m+k}$, $k\in\N$, if and only if there is a finite positive Borel measure $\nu$ on
$\R_+$ of the total mass $1$ such that
  \begin{equation}\label{twoschoen1}
\nu_m(dx)=p_m(x)\,dx, \qquad p_m(x) = \frac{2x^{m-1}}{B\bigl(\frac{m}2,\frac{k}2\bigr)}\,\int_x^\infty
\Bigl(1-\frac{x^2}{u^2}\Bigl)^{\frac{k}2-1}\,\frac{\nu(du)}{u^m}.
  \end{equation}
In this case $\nu=\nu_{m+k}$. In particular, $\nu_m$ is absolutely continuous  and $\nu_m\{(0,\varepsilon)\}>0$
for any $\varepsilon>0$.
   \end{theorem}
%%%%%%%%%%%%%%%%%%%%%%%%%%%%%%%%%%%%%%%%%%%%%%%
  \begin{proof}
Assume that $f\in\Phi_{m+k}$. Then $f$ admits two representations
\begin{equation}\label{twoschoen}
f(r)=\int_0^\infty \Omega_m(ru)\nu_m(du)=\int_0^\infty \Omega_{m+k}(ru)\nu_{m+k}(du).
  \end{equation}
It is not hard to obtain a relation between the measures $\nu_m$ and $\nu_{m+k}$. Indeed, recall Sonine's
integral \cite[formula (4.11.11), p. 218]{AAR}
 $$
J_\mu(t)=\frac2{\Gamma(\mu-\lambda)}\,\Bigl(\frac{t}2\Bigr)^{\mu-\lambda}\,
\int_0^1 J_\lambda(ts)s^{\lambda+1}(1-s^2)^{\mu-\lambda-1}\,ds,
\qquad \mu>\lambda\ge-\frac12.
  $$
For the values
  $$
\lambda=\frac{m}2-1, \qquad \mu=\frac{m+k}2-1=\lambda+\frac{k}2
 $$
one has in terms of $\Omega$'s
\begin{equation}\label{sonine}
 \Omega_{m+k}(t)=\frac2{B\bigl(\frac{m}2,\frac{k}2\bigr)}\,\int_0^1 \Omega_m(ts)s^{m-1}(1-s^2)^{\frac{k}2-1}\,ds.
 \end{equation}
We plug the latter equality into \eqref{twoschoen} to obtain
  \begin{equation*}
\begin{split}
f(r) &=\frac2{B\bigl(\frac{m}2,\frac{k}2\bigr)}\,\int_0^\infty\nu_{m+k}(du)\int_0^1
\Omega_m(rus)s^{m-1}(1-s^2)^{\frac{k}2-1}\,ds \\
&=\frac2{B\bigl(\frac{m}2,\frac{k}2\bigr)}\,\int_0^\infty \frac{\nu_{m+k}(du)}{u^{m}}
\int_0^u\Omega_m(rx)x^{m-1}\Bigl(1-\frac{x^2}{u^2}\Bigl)^{\frac{k}2-1}\,dx \\
&= \frac{2x^{m-1}}{B\bigl(\frac{m}2,\frac{k}2\bigr)}\,\int_0^\infty
\Omega_m(rx)dv\int_x^\infty \Bigl(1-\frac{x^2}{u^2}\Bigl)^{\frac{k}2-1}\,
\frac{\nu_{m+k}(du)}{u^{m}}\,.
\end{split}
   \end{equation*}
Due to the uniqueness of Schoenberg's representation we arrive at \eqref{twoschoen1}.

Conversely, starting from \eqref{twoschoen1} and reversing the argument we come to \eqref{twoschoen}
with $\nu_{m+k}=\nu$.

Since $f \not = const$,  $\nu_{m+k}\not=\delta_0$, we see that $\nu_m$ is absolutely
continuous, and $\nu_m\{(0,\varepsilon)\}>0$ for any $\varepsilon>0$, as claimed.
\end{proof}

\begin{remark}
A closely related result is obtained in \cite[Theorem 6.3.5]{TriBel} where a certain condition for
$f\in\Phi_n$ to belong to $\Phi_m$ with $m>n$ is given in terms of $f$ itself.
\end{remark}

We call \eqref{twoschoen1} the {\it $k$-step transition formula}.

We can paraphrase the statement of Theorem \ref{prop_measure_relation} as follows:
$f\in\Phi_m\backslash\Phi_{m+1}$ as long as $\nu_m$ is either not pure absolutely continuous or $\nu_m\{(0,\varepsilon)\}=0$ for some $\varepsilon>0$.

\begin{corollary}
Let $f\in\Phi_m$, and let its Schoenberg's measure $\nu_m$ be a pure point one. Then
$f\in\Phi_m\setminus\Phi_{m+1}$. In particular,
    \begin{equation}
f(t)=\sum^{\infty}_{k=1}\alpha_k\Omega_n(r_k t)\in\Phi_n\backslash\Phi_{n+1}.
    \end{equation}
for any sequence of nonnegative numbers $\{r_k\}_{k\ge1}$ and $\{\alpha_k\}_{k\ge1}\in l^1(\Bbb N)$.
  \end{corollary}

\begin{example}
Let
$$ f_1(r)=e^{-r}, \qquad f_2(r)=e^{-r^2}. $$
Both functions belong to $\Phi_n$ for all $n\in\N$, and their Schoenberg's measures are known
explicitly, see, e.g., \cite[Chapter 1]{Stein-Weiss},
\begin{equation}
\nu_m(f_1)=\frac2{B\bigl(\frac{m}2, \frac12\bigr)}\,\frac{u^{m-1}}{(1+u^2)^{\frac{m+1}2}}\,du, \qquad
\nu_m(f_2)=\frac1{\Gamma(\frac{m}2)} \left(\frac{u}2\right)^{m-1}\,\exp\left\{-\frac{u^2}4\right\}\,du.
\end{equation}
It is a matter of simple (though lengthy) computations to verify formula \eqref{twoschoen1} for
each of these sequences of measures for all $m,k\in\N$.
\end{example}

Let us single out the simplest case $k=2$.
\begin{corollary}
$f\in\Phi_m$ belongs to $\Phi_{m+2}$ if and only if there is a finite positive Borel measure $\nu$ on
$\R_+$ of the total mass $1$ such that
  \begin{equation}\label{twoschoen2}
\nu_m(dx)=p_m(x)\,dx, \qquad p_m(x) = mx^{m-1}\,\int_x^\infty \frac{\nu(du)}{u^m},
\quad \nu=\nu_{m+2}.
\end{equation}
\end{corollary}

\medskip

The problem we address now concerns the smoothness properties and the rate of decay of measures
$\nu_m$ in \eqref{twoschoen1} for the case $k=2j$, $j\in\N$.

We start with the following auxiliary statement. Let  $C^{k}(\R_+)$ denote the Freshet
space of $k$-smooth continuous functions defined on an open set $\R_+ =(0,\infty)$, and
let $AC[0,a]$ be the space of absolutely continuous functions on $[0,a]$.  We also put
$$
AC_{\loc}(\Bbb R_+):= \{f\in AC[0,a]\ \forall a>0\}, \quad AC^k_{\loc}(\R_+) := \{f\in
C^{k-1}(\R_+):\ f^{(k-1)}\in AC_{\loc}(\R_+)\}.
$$
%
%
%%%%%%%%%%%%%%%%%%%%%%%%%%%%%%%%%%%%%%%%%%%%%%%%%%%%%%%%%%%%%%%%%%%%%%%%%%%%%%%%%%%%%%%%%%
   \begin{lemma}\label{Differentian_lemma}
Let  $\sigma$ be a function of bounded variation on $\R_+$, 
$\varphi(x,\cdot)\in AC_{\loc}(\Bbb R_+)$ for each $x\in\Bbb R_+,$ 
$D_t\varphi(\cdot,t)\in C^k(0,\infty)$ for each $t\in\Bbb R_+,$ and let
    \begin{equation}\label{3.7}
g(x) := \int^{\infty}_x \varphi (x,t)d\sigma(t).
    \end{equation}
Assume also that there exist functions $\zeta_0\in L^1([\varepsilon, \infty);d\sigma)$ and
$\psi_j\in L^1[\varepsilon, \infty)$ for each $\varepsilon>0$, and such that
  \begin{equation}\label{4.18AA}
|\varphi(x,t)| \le \zeta_0(t),\quad |D^j_x D_t\varphi(x,t)| \le \psi_j(t),\quad x\in
\R_+,  \quad  j=0,1,\ldots,k-1.
    \end{equation}
Moreover, assume that
  \begin{equation}\label{4.18}
\varphi(x,x)=0\quad \text{and}\quad  D^j_x D_t\varphi(x,t)|_{t=x}=0, \quad
j0,1,\ldots,k-2,
    \end{equation}
and  $\lim_{t\to\infty} \varphi(x,t)\sigma(t)=0$ for each $x\in \R_+$. Then $g\in C^k(\R_+)$. 
If in addition
  \begin{equation}\label{4.18A}
\lim_{t\to\infty}D^k_x \varphi(x,t)=0 \quad \text{for each} \ x\in \R_+,
  \end{equation}
then
   \begin{equation}\label{4.19}
g^{(k)}(x) = \sigma(x)\left([D^{k-1}_x (D_x + D_t)\varphi(x,t)]|_{t=x}\right) +
\int^{\infty}_x D_x^k\varphi (x,t)d\sigma(t).
   \end{equation}
   \end{lemma}
%%%%%%%%%%%%%%%%%%%%%%%%%%%%%%%%%%%%%%
\begin{proof}
Since $\varphi(x,\cdot)\in A C_{\loc}(\Bbb R_+)$ and $\varphi(x,\cdot) \in L^1([x,
\infty);d\sigma)$  for each $x\in\Bbb R_+,$ one gets after integrating by parts with
account of   the first relation in \eqref{4.18}
   \begin{equation}\label{3.12}
g(x) = - \varphi(x,x)\sigma(x) - \int^{\infty}_x \varphi'_{t}(x,t)\sigma(t){dt} = -
\int^{\infty}_x \varphi'_{t}(x,t)\sigma(t){dt}.
  \end{equation}
On the other hand, since $\sigma$ is bounded,  $\psi_j\sigma \in L^1[\varepsilon, \infty)$ 
for each $\varepsilon>0$ and $j=0,1,\ldots,k-1$, and
according to \eqref{4.18AA}  
$$ (D^j_x D_t\varphi(x,t))\sigma(t) \le \psi_j(t)\sigma(t), \quad x\in\R_+. $$
Therefore one can differentiate \eqref{3.12} subsequently with account
of \eqref{4.18} to obtain by induction  that $g\in C^j(0,\infty)$ and
   \begin{equation}
g^{(j)}(x) = %%-\frac{d}{dx}[ \varphi (x,x)\sigma(x)] + \varphi'_t(x,t)|_{t=x}\sigma(x)
- \int^{\infty}_x D^j_x D_t\varphi(x,t)\sigma(t){dt}, \quad j\in\{0,1,\ldots,k-1\},
  \end{equation}
and
   \begin{equation}\label{3.14A}
g^{(k)}(x) = \sigma(x)\left(D^{k-1}_x D_t\varphi(x,t)|_{t=x}\right)  - \int^{\infty}_x
D^k_x D_t\varphi(x,t)\sigma(t){dt}.
  \end{equation}
Integrating   identity  \eqref{3.14A} by parts  and taking \eqref{4.18A} into account we
arrive at \eqref{4.19}.
    \end{proof}
%%%%%%%%%%%%%%%%%%%%%%%%%%%%%%%%%%%%%%
%%%%%%%%%%%%%%%%%%%%%%%%%%%%%%%%%%%%%%%
%
   \begin{remark}
Note that we do not assume the majorant $\zeta_0\in L^1[\varepsilon, \infty)$ for each $\varepsilon>0$. 
The existence of the Lebesgue
integral in \eqref{3.12} is implied by estimate \eqref{4.18AA} with $j=0$.
\end{remark}
%
%%%%%%%%%%%%%%%%%%%%%%%%%%%%%%%%%%%%%%%%%%%%%%%%
   \begin{corollary}\label{homogeneous_different_cor}
Assume that $\varphi_0\in AC^k[0,1]$ and
   \begin{equation}\label{3.15}
\varphi(x,t) := \frac{1}{t}\varphi_0\left(\frac{x}{t}\right),\  t\ge
x,\quad\text{and}\quad \varphi_0(1)=\varphi'_0(1)=\ldots =\varphi_0^{(k-1)}(1)=0.
   \end{equation}
Let also $\sigma(\cdot)$ be of bounded variation and let $g(\cdot)$ be given by
\eqref{3.7}. Then
   \begin{equation}\label{4.19AA}
g^{(k)}(x) =  \int^{\infty}_x D_x^k\varphi (x,t)d\sigma(t).
   \end{equation}
  \end{corollary}
%%%%%%%%%%%%%%%%%%%%%%%%%%%%%%%%%%%%%%%%%%%%%
    \begin{proof}
Clearly, conditions \eqref{4.18} are implied by conditions \eqref{3.15}. Moreover,
estimates  \eqref{4.18AA}  hold with $\psi_j(t)= C_j t^{-(j+1)}$ and
$C_j=\|\varphi_0^{(j)}\|_{C[0,1]}$, $j\in\{1,\ldots,k\}$. Besides $|\varphi(x,t)|\le C_0
t^{-1}$, where $C_0=\|\varphi_0\|_{C[0,1]}$, hence $\varphi(x,t)\in L^1([x,\infty);
d\sigma)$ for each $x\in \R_+$.

Further,  it is easily seen that
   \begin{eqnarray}
(D_x + D_t)D^{k-1}_x\varphi(x,t) = (D_x+D_t)\left(\frac{1}{t^k}
\varphi_0^{(k-1)}\left(\frac{x}{t}\right)\right) \nonumber \\
= (-1)^{k-1} \left[\frac{1}{t^{k+1}}\varphi_0^{(k)}\left(\frac{x}{t}\right) \left[1 -
\frac{x}{t}\right] - \frac{k}{t^{k+1}}\varphi_0^{(k-1)}\left(\frac{x}{t}\right)
\right]\,.
%%\frac{x}{t^{k+2}}\varphi_0^{(k)}\left(\frac{x}{t}\right) \right]
   \end{eqnarray}
It follows with account of \eqref{3.15} that
    \begin{equation}\label{3.18}
[D^{k-1}_x (D_x + D_t)\varphi(x,t)]|_{t=x} =
(-1)^k\frac{k}{x^{k+1}}\varphi^{(k-1)}_0(1)=0.
   \end{equation}
Thus conditions of Lemma  \ref{Differentian_lemma} are met and \eqref{4.19AA}  follows
by combining  \eqref{4.19}  with  \eqref{3.18}.
        \end{proof}
%%%%%%%%%%%%%%%%%%%%%%%%%%%%%%%%%%%%%%
%

Recall a classical result of Hardy, Littlewood and P\'olya \cite[Theorem 319]{hlp}.

{\bf Theorem HLP}. Let $T$ be a measurable function on $\R^2_+$ such that it is homogeneous
of the degree $-1$, that is, $T(\lambda x,\lambda t)=\lambda^{-1}T(x,t)$ for $\lambda>0$,
and for some $p$, $1\le p<\infty$
\begin{equation}\label{hlpnorm}
\tau_p:=\int_0^\infty |T(1,t)|t^{-1/p}\,dt<\infty.
\end{equation}
Then the integral operator generated by $T$
$$ (Th)(x):=\int_0^\infty T(x,t)h(t)\,dt $$
is bounded in $L^p(\R_+)$, and its norm $\|T\|\le\tau_p$.

A typical example which will be of particular interest for us is
\begin{equation}
T(x,t)=\left\{
         \begin{array}{ll}
           \frac1{t}\,T\bigl(\frac{x}{t}\bigr), & \hbox{$t\ge x$;} \\
           $0$, & \hbox{$t<x$,}
         \end{array}
       \right.
\end{equation}
where $T$ is a polynomial. Now
$$ \tau_p=\int_0^1|T(u)|u^{1/p-1}\,du<\infty, \quad p\ge1. $$
   \begin{theorem}\label{prop_smoothness}
Let  $(const\not=)f\in\Phi_{m}$ with the  Schoenberg's measure $\nu_m$, $m\in\N$. Assume that $f\in\Phi_{m+2j}$, $j\in\N,$ with the  Schoenberg's measure $\nu=\nu_{m+2j}$. Then the following relations hold
  \begin{enumerate}
  \item [\em (i)]
  $\nu_m\in AC^{j}_{\loc}(\R_+)$.

  \item [\em (ii)]  If $\nu$ is absolutely continuous, $\nu=p_{m+2j}\,dx$, and
$p_{m+2j}\in L^p(\Bbb R_+)$ for some $1\le p<\infty$, then
    \begin{equation}\label{4.23A}
x^k p^{(k)}_m(x)\in L^p(\Bbb R_+), \qquad   k=0,1,\ldots, j-1.
    \end{equation}
Moreover,
   \begin{equation}\label{4.24B}
p^{(k)}_m(x)  = o\left(\frac{1}{x^{k+1}}\right), \quad x\to\infty, \qquad  
k=0,1,\ldots, j-1.
   \end{equation}
  \item [\em (iii)]
Let $\nu$ satisfy
  \begin{equation}\label{moment}
I:=\int_0^\infty \frac{\nu(du)}{u}<\infty.
  \end{equation}
Then
  \begin{equation}\label{atzero}
p_m(0)=\lim_{x\to 0+}p_m(x)=0, \quad m\ge2, \qquad p_1(0)=I.
  \end{equation}
Furthermore,
    \begin{equation}\label{4.23AB}
x^k p^{(k)}_m(x)\in L^\infty(\Bbb R_+), \qquad   k=0,1,\ldots j-1.
\end{equation}
  \end{enumerate}
     \end{theorem}
%%%%%%%%%%%%%%%%%%%%%%%%%%%%%%%%%%%%%%%%%
        \begin{proof}
(i) For $m,j\in\N$ we put
      \begin{equation}\label{4.24A}
P_{m,j}(u):=\frac{u^{m-1}(1-u^2)^{j-1}}{B\bigl(\frac{m}2,j \bigr)}, \quad
Q_{m,j}(x,t) := \frac1{t}\,P_{m,j}\left(\frac{x}{t}\right)=\frac{2x^{m-1}}{B\bigl(\frac{m}2,j \bigr){t^m}}
\left(1-\frac{x^2}{t^2}\right)^{j-1}\,,
      \end{equation}
and
    \begin{equation}\label{4.26}
Q_{m,j,k}(x,t) := x^k D^k_x Q_{m,j}(x,t) = x^k
t^{-k-1}P_{m,j}^{(k)}\left(\frac{x}{t}\right), \quad k \in \{1,2\ldots,j-1\}, \quad
Q_{m,j,0}=Q_{m,j}.
    \end{equation}
It is easily seen that the kernel  $\varphi=Q_{m,j}$ is
homogeneous of degree $-1$ and meets the hypothesis of Corollary \ref{homogeneous_different_cor}.
So, Corollary \ref{homogeneous_different_cor}  applied to relation \eqref{twoschoen1} provides
$p_m~\in AC^{j-1}_{\loc}(\R_+)$ and
\begin{eqnarray}\label{4.26A}
x^k p^{(k)}_m(x) = \int^{\infty}_x Q_{m,j,k}(x,t)\,\nu(dt), \qquad k=0,1, \ldots, j-1.
\end{eqnarray}

(ii). It is clear that the kernels $Q_{m,j,k}$, $k \in \{0,\ldots,j-1\}$, are
homogeneous of the degree $-1$, and \eqref{hlpnorm} holds for all $1\le p<\infty$. By
HLP theorem, the integral operators
  \begin{equation*}
(Q_{m,j,k}h)(x)=\int_x^\infty Q_{m,j,k}(x,t)h(t)\,dt, \qquad k=0,\ldots,j-1,
 \end{equation*}
are bounded in $L^p(\R_+)$ for all $1\le p<\infty$. The latter relation with $h=
p_{m+2j}$ leads directly to \eqref{4.23A} in view of \eqref{4.26A}. Note that
$p_{m+2j}\in L^1(\R_+)$ (so \eqref{4.23A} holds automatically for $p=1$) since
$\nu$ is a finite measure.

Next, it easily follows from \eqref{4.24A} and \eqref{4.26} that
   \begin{equation}\label{4.31}
\sum_{k=0}^{j-1}|Q_{r,m,k}(x,t)|\le \frac{C_{m,j}}{t}, \quad t\ge x
   \end{equation}
and since the measure $\nu$ is finite, we have from \eqref{4.26A}
   \begin{equation}\label{4.26AB}
|x^k p^{(k)}_m(x)| \le C_{m,j}\int^{\infty}_x \frac{\nu(dt)}{t} \le
\frac{C_{m,j}}{x}\int^{\infty}_x \nu(dt)= o\left(\frac{1}{x}\right), \quad x\to\infty
   \end{equation}
for $1\le k\le j-2$.  For $k = j-1$ the estimate is a consequence of
\eqref{4.31}, \eqref{4.24A} and \eqref{4.26A}, and so \eqref{4.24B} follows.

(iii).  Limit relations \eqref{atzero} follow easily  from \eqref{twoschoen1}, the bound
$$
0 < \left(1-\frac{x^2}{t^2}\right)^{\frac{k}2-1}\,\left(\frac{x}{t}\right)^{m-1}<1,
\qquad t\ge x,
$$
and the Dominated Convergence Theorem in view of assumption \eqref{moment}.

Relations \eqref{4.23AB} arise directly from \eqref{atzero} and \eqref{4.24B}.
The proof is complete.
         \end{proof}
%%%%%%%%%%%%%%%%%%%%%%%%%%%%%%%%%%%
%%%%%%%%%%%%%%%%%%%%%%%%%%%%%%%%%%%%%%%%%%%%%%%%%%%%

We turn now to the case $k=\infty$ in Theorem \ref{prop_measure_relation}. Recall that
$$ \Phi_\infty := \bigcap^{\infty}_{m=1}\Phi_m. $$
A theorem of Schoenberg states that $f\in\Phi_\infty$ if and only if there is a positive measure
$\sigma$ on $\R_+$ with the total mass $1$ such that
\begin{equation}\label{schninf}
f(r)=\int_0^\infty e^{-sr^2}\,\sigma(ds).
\end{equation}

\begin{proposition}\label{schoeninfty}
Let  $(const\not=)f\in\Phi_{m}$, and let $\nu_{m}$ be its Schoenberg's measure. Then $f\in\Phi_\infty$
if and only if there is a finite positive Borel measure $\sigma$ on $\R_+$ of the total mass $1$ such that
  \begin{equation}\label{schinf1}
\nu_m(dx)=p_{m,\sigma}(x)\,dx, \qquad p_{m,\sigma}(x) := \frac{x^{m-1}}{2^{\frac{m}2-1}\Gamma(\frac{m}2)}
\int_0^\infty (2s)^{-m/2}\exp\biggl(-\frac{x^2}{4s}\biggr)\,\sigma(ds).
  \end{equation}
The density $p_{m,\sigma}$ can be extended as an analytic function to the sector $\{|\arg x|<\frac{\pi}4\}$.
   \end{proposition}
\begin{proof}
We argue in the same manner as in Theorem \ref{prop_measure_relation}. The only difference is that  we use \eqref{schninf}
instead of Schoenberg's representation in $\Phi_{m+k}$, and the equality
\begin{equation*}
e^{-sr^2}=\frac{1}{2^q\Gamma(q+1)}\,\int_0^\infty\Omega_n(ru)\frac{u^{n-1}}{(2s)^{n/2}}\,
\exp\biggl(-\frac{u^2}{4s}\biggr)\,du
\end{equation*}
in place of Sonine's integral.
\end{proof}

\begin{remark}
Let $f\in\Phi_m$ with Schoenberg's measure $\nu_m$ \eqref{schinf1}. It is clear that \eqref{schninf} holds
then for $f$. The following problem arises naturally. Given $f\in\Phi_\infty$, whether it is possible
to obtain \eqref{schinf1} as a limit case of \eqref{twoschoen1} as $k\to\infty$, proving thereby the
result of Schoenberg \eqref{schninf}?
\end{remark}

%%%%%%%%%%%%%%%%%%%%%%%%%%%%%%%%%%%%%%%%%%%%%%%%%%%%%%%%%%%%%%%%%%%%%%%%%%%%%%%%%%%%%%%%%%
  \begin{proposition}
Let $f\in\Phi_m$ with $m\ge 3$, i.e., $f$ admits representation   \eqref{schoenberg1}
with the measure  $\sigma_m$.
Then it admits the representation
    \begin{equation}\label{4.22A}
f(r)=\int^{\infty}_0\Omega_{m-2}(rx)\,\sigma_{m-2}(dx)
    \end{equation}
with the measure $\sigma_{m-2}$ given by
    \begin{equation}\label{4.23}
\sigma_{m-2}(dx) = \frac{1}{2}\int^x_0 u^{m-3}du\int^{\infty}_u\frac{\sigma_m(dt)}{t^{m-2}}\, dx.
      \end{equation}
Conversely, if $f$ admits representation \eqref{4.22A} with  $\sigma_{m-2}$ of the form \eqref{4.23},
then $f\in\Phi_m$.
  \end{proposition}
%%%%%%%%%%%%%%%%%%%%%%%%%%%%%%%%%%%%%%%%%%%%%%%%%%%%%%%%%%%%%%%%%%%%%%%%%
       \begin{proof}
Our considerations rely on the following identity \cite[Section III.3.2]{Wat}
    \begin{eqnarray}\label{4.24}
\frac{d}{dx}\Bigl((r x)^{\frac{m-2}{2}}J_{\frac{m-2}{2}}(r x)\Bigr)
=r(rx)^{\frac{m-2}{2}}J_{\frac{m-4}{2}}(rx).
    \end{eqnarray}
We let
    \begin{equation}\label{4.25}
\widehat{\sigma}_m(x) = - \int^{\infty}_x\frac{\sigma_m(dt)}{t^{m-2}}\le 0.
    \end{equation}
Recall that
   \begin{equation}
J_{\nu}(x)= O(x^{\nu}),\quad  x\to 0 \quad \text{and}\quad J_{\nu}(x)=
O\Bigl(\frac{1}{\sqrt{x}}\Bigr), \quad x\to\infty.
   \end{equation}
Using the first of these relations  and applying the dominated convergence theorem to
the function  $(x/t)^{m-2}$ as $x\to 0$ with the majorant $1\in L^1(\Bbb
R_+,d\sigma_m)$, we get from \eqref{4.25}
    \begin{eqnarray}\label{4.28}
   \begin{split}
&{} \lim_{x\to 0} x^{\frac{m-2}{2}}J_{\frac{m-2}{2}}(r x)\widehat\sigma_m(x)=\lim_{x\to 0}
x^{m-2}\widehat\sigma_m(x)\\
&= - \lim_{x\to 0}\int^{\infty}_x\left(\frac{x}{t}\right)^{m-2}\sigma_m(dt)=0, \qquad
m\ge 2.
\end{split}
    \end{eqnarray}
Further, $\widehat \sigma_m(x) = o(\frac{1}{x^{m-2}})$. Therefore,
    \begin{equation}\label{4.29A}
\lim_{x\to \infty} x^{\frac{m-2}{2}}J_{\frac{m-2}{2}}(r x)\widehat\sigma_m(x) =
\lim_{x\to \infty} x^{\frac{m-3}{2}} \widehat\sigma_m(x) =0.
 \end{equation}
Transformation of \eqref{schoenberg1} in view of \eqref{4.25} and integration by parts taking
\eqref{4.24},  \eqref{4.28}, and \eqref{4.29A} into account give
    \begin{equation}\label{6}
    \begin{split}
f(r) &=\int^{\infty}_0 \left(\frac{2}{rx}\right)^{\frac{m-2}{2}}J_{\frac{m-2}{2}}(rx)\sigma_m(dx)
= \frac{2^{\frac{m-2}{2}}}{r^{m-2}}\,\int^{\infty}_0\frac{(rx)^{\frac{m-2}{2}}J_{\frac{m-2}{2}}(rx)}
{x^{m-2}}\, {\sigma}_m(dx) \\  &=\frac{2^{\frac{m-2}{2}}}{r^{m-2}}\int^{\infty}_0(rx)^{\frac{m-2}{2}}J_{\frac{m-2}{2}}(rx)\widehat{\sigma}_m(dx) \\
&= \frac{2^{\frac{m-2}{2}}}{r^{m-2}}(rx)^{\frac{m-2}{2}}J_{\frac{m-2}{2}}(rx)\widehat{\sigma}_m(x) \big|_0^\infty
 - \frac{2^{\frac{m-2}{2}}}{r^{m-2}}\int^{\infty}_0
\widehat{\sigma}_m(x) \frac{d}{dx} \bigl((rx)^{\frac{m-2}{2}}J_{\frac{m-2}{2}}(rx)\bigr)dx \\
%&= - \frac{2^{\frac{m-2}{2}}}{r^{m-2}}
%r\int^{\infty}_0 (rx)^{\frac{m-2}{2}}J_{\frac{m-4}{2}}(rx)\cdot \widehat{\sigma}_m(x) dx \nonumber \\
&= - \frac{2^{\frac{m-2}{2}}}{r^{m-3}}
\int^{\infty}_0 r^{m-3}\left[(rx)^{\frac{4-m}{2}}J_{\frac{m-4}{2}}(rx)\right]x^{m-3}\widehat{\sigma}_m(x)\,dx.
\end{split}
    \end{equation}
Setting
    \begin{equation}\label{4.30}
\sigma_{m-2}(x):= -\frac{1}{2}\int^x_0 t^{m-3}\widehat{\sigma}_m(t)dt \ (\ge 0),
    \end{equation}
we note that $\sigma_{m-2}$ is non-negative increasing bounded function. Indeed,
    \begin{equation}
    \begin{split}
\sigma_{m-2}(\infty) &=  \frac{1}{2}\int_0^\infty
x^{m-3}dx\int^{\infty}_x\frac{\sigma_m(dt)}{t^{m-2}} \\
&=\frac{1}{2}\int^{\infty}_0\frac{\sigma_m(dt)}{t^{m-2}}\,\int^t_0 x^{m-3}dx=\frac{1}{2(m-2)}\int^{\infty}_0 \sigma_m(dt)<\infty.
\end{split}
  \end{equation}
Therefore we can rewrite representation \eqref{6} in the form \eqref{4.22A}.
       \end{proof}

%%%%%%%%%%%%%%%%%%%%%%%%%%%%%%%%%%%%%%%%%%%%%%%%%%%%%%%
  \begin{corollary}
Let $f \in\Phi_1$, i.e.
         \begin{equation}\label{4.29}
f(r)= \int^{\infty}_0 \cos(r x)d\sigma_1(x)
         \end{equation}
Then $f \in \Phi_3$  if and only if $\sigma_{1}$ is upper convex. In the later case it admits
a representation
    \begin{equation}\label{4.22}
f(r) = \int^{\infty}_0\Omega_3(r x) d\sigma_{3}(x)
   = \int^{\infty}_0\frac{\sin(r x)}{rx}d\sigma_{3}(x)
    \end{equation}
where
    \begin{equation}\label{4.34}
 \sigma_3(x) = \sigma_3(0) - \int^x_0 t d\sigma'_1(t)d.
   \end{equation}
In particular, $\sigma_3$ is a monotonically decreasing, bounded function, which is absolutely
continuous with respect to $\sigma'_1$. In fact,  the measures $d \sigma_3$  and
$d \sigma'_1$ are equivalent:  $d\sigma'_1(x) = - \frac{1}{2x}d\sigma_3(x)$.
    \end{corollary}
%%%%%%%%%%%%%%%%%%%%%%
\begin{proof}
\emph{Necessity.} It follows from \eqref{4.30} with $m=3$ that
    \begin{equation}\label{10}
\sigma_1(x)=\frac{1}{2}\int^x_0\bigl(-\widehat\sigma_3(t)\bigr)dt.
    \end{equation}
On the other hand, according to \eqref{4.25} $\widehat\sigma_3(\ge 0)$ decreases.
Hence $\sigma_1(\cdot)$ is upper convex  (see \cite{Nat74}).

\emph{Sufficiency.} Let $\sigma_1$ in representation \eqref{4.29} be upper
convex. Then it is locally absolutely continuous, hence
 admits a representation $\sigma_1(x) = \sigma_1(0) + \int^x_0 d\sigma'_1(t)dt$
 with an  increasing function $\sigma'_1(\cdot)$ (see \cite{Nat74}).
%%$-\widehat\sigma_3(t) = 2\sigma'_1(t)$.
Since $\sigma_1$  monotonically increases and is bounded, the Fatou theorem
applies and gives  $\int^{\infty}_0\sigma'_1(x)dx \le \sigma_1(+\infty) - \sigma_1(0)
<\infty$. Using  decaying of  $\sigma'_1$  we get
  \begin{equation}\label{4.36}
0\le\lim_{x\to\infty}x\sigma'_1(x)\le\lim_{x\to\infty}2\int^x_{x/2}\sigma'_1(t)dt=0.
    \end{equation}
Integration by parts in \eqref{4.29} in view of \eqref{4.36}, definition
\eqref{4.34}, and the inclusion $\sigma_1\in AC_{\loc}(\R_+)$ gives
     \begin{equation}
     \begin{split}
f(r) &= \int^{\infty}_0 \cos(r x)\sigma'_1(x)dx =
\frac{1}{r}\int^{\infty}_0\sigma'_1(x)d\sin (rx) \\ &=
\frac{\sin r x}{r}\sigma'_1(x)\big |^{\infty}_0 - \int^{\infty}_0\frac{\sin(r x)}{rx} x d\sigma'_{1}(x)
= \int^{\infty}_0\frac{\sin(r x)}{rx}d\sigma_{3}(x).
\end{split}
      \end{equation}
It remains to show that $\sigma_{3}(\infty)<\infty$. Integrating by parts in
\eqref{4.34} and tending $x\to\infty$  with account of \eqref{4.36}, we derive
$$
\sigma_{3}(\infty) - \sigma_{3}(0) =  - \int_0^{\infty} t d\sigma'_1(t) = -t
\sigma'_1(t)\big |^{\infty}_0 + \int_0^{\infty} \sigma'_1(t)dt \le \sigma_{1}(\infty) -
\sigma_{1}(0) < \infty.
$$
Thus $\sigma_{3}(\cdot)$ is bounded and the inclusion  $f \in \Phi_3$ is proved.
     \end{proof}
%%%%%%%%%%%%%%%%%%%%%%%%%%%%%%%%%%%%%%%%%%
     \begin{remark}
In the case $k=2$ equation \eqref{4.23}  is equivalent to the differential equation
   \begin{equation}
\frac{d}{d\sigma_m}\left(\frac{1}{x^{m-3}}\frac{d}{dx}\right)\sigma_{m-2}(x)=-\frac{1}{2x^{m-2}},
\quad m\in \N,
  \end{equation}
subject to certain boundary conditions. For $m=3$ this equation is just Krein's string
equation with respect to the unknown monotone function $\sigma_{m-2}$ and the
mass distribution function $\sigma_m$.
  \end{remark}
%%%%%%%%%%%%%%%%%%%%%%%%%%%%%%%%%%%%%%%%%%%%%%%%%%
%%%%%%%%%%%%%%%%%%%%%%%%%%%%%%%%%%%%%%%%%%%%%%

\section{Products of Schoenberg's kernels}

We demonstrate here a power and capability of the transition formula from Theorem \ref{prop_measure_relation}
by applying it to products and squares of the Schoenberg kernels $\Omega_n$.

   \begin{theorem}[=Theorem \ref{schoenkern}, (i)]
For any pair of points $\{a,b\}\subset\Bbb R_+$, $a\not= b$, the function
 $$
\Omega_n(a,b; t) := \Omega_n(a t)\cdot \Omega_n(b t) \in\Phi_n\backslash\Phi_{n+1}.
  $$
  \end{theorem}
%%%%%%%%%%%%%%%%%%%%%%%%%%%%%%%%%%%%%%%
   \begin{proof}
The case $ab=0$ is trivial, so we let $a,b>0$. Clearly, $\Omega_n(a,b)\in\Phi_n$ since the Schur (entrywise)
product of two nonnegative matrices is again a nonnegative matrix. Next, by \eqref{fouriersphere}
$\Omega_n(a,b; |\cdot|)$ is the Fourier transform of the convolution of the Lebesgue
measures on the spheres $S^{n-1}_{a}$ and $S^{n-1}_{b}$. As is well known (see, e.g., \cite[Theorem 6.2.3]{LO}) the
support of a convolution equals a closure of the algebraic sum of the supports for the components, so the support
of the convolution in question is the spherical annulus $\{x\in\R^n: |a-b|\le |x|\le a+b\}$.
The Schoenberg measure of $\Omega_n(a,b;|\cdot|)$  comes up as the spherical projection of the above convolution of
two Lebesgue measures on the spheres, so its support is the interval $[|a-b|, a+b]$, disjoint
from the origin (for an explicit expression of this measure see \cite{ost}). An application of
Theorem \ref{prop_measure_relation} completes the proof.
  \end{proof}

\begin{remark}
Given two functions $f_1,f_2\in\Phi_n\backslash\Phi_{n+1}$, let their Schoenberg's measures $\nu(f_1)$ and $\nu(f_2)$ have
disjoint supports
$$ \supp\nu(f_j)\subset [a_j,b_j], \quad j=1,2, \qquad a_1<b_1<a_2<b_2. $$
The same argument shows that $f_1f_2\in\Phi_n\backslash\Phi_{n+1}$.
\end{remark}

The case $a=b$ is much more delicate.

\begin{theorem}[=Theorem \ref{schoenkern}, (ii)] \label{square}
$\Omega_n^2\in\Phi_{2n-1}\backslash\Phi_{2n}$,\  $n\in\N$.
\end{theorem}
%%%%%%%%%%%%%%%%%%%%%%%%%%%%%%%%%%%
\begin{proof}
Assume first that $n\ge2$. We begin with the known formula from the Hankel transforms theory  (see, e.g., \cite[(22), p.24]{intran})
\begin{equation}
\int_0^2 J_\nu(tx)\sqrt{tx}\,\frac{dx}{\sqrt{x(4-x^2)}} = \frac{\pi}2\sqrt{t}\,J_{\nu/2}^2(t), \qquad \nu>-1.
\end{equation}
Now put $\nu=n-2$, so in terms of $\Omega$'s
$$ \Omega_n(y)=\Gamma\Bigl(\frac{n}2\Bigr)\left(\frac2{y}\right)^{\frac{n}2-1}J_{\frac{n}2-1}(y), \qquad
\Omega_{2n-2}(y)=\Gamma(n-1)\left(\frac2{y}\right)^{n-2}J_{n-2}(y), $$
and hence
\begin{equation}\label{schoentran}
\Omega_n^2(t)=\int_0^2\Omega_{2n-2}(tx)\,p_{2n-2}(x)dx, \qquad
p_{2n-2}(x)=C_n\,\frac{x^{n-2}}{\sqrt{4-x^2}}, \quad C_n=\frac{2\Gamma^2\bigl(\frac{n}2\bigr)}{\pi\Gamma(n-1)}\,.
\end{equation}
We see thereby that $f=\Omega_n^2\in\Phi_{2n-2}$.

The algorithm we apply now may be called ``one step backward and two steps forward''. First we find the Schoenberg measure $\nu_{2n-3}(f)$
from the 1-step transition formula \eqref{twoschoen1} with $m=2n-3$
\begin{equation*}
p_{2n-3}(x)=x^{2n-4}\,\int_x^2\frac{p_{2n-2}(u)du}{u^{2n-4}\sqrt{u^2-x^2}}=C_nx^{2n-4}\,\int_x^2\frac{u^{2-n}du}{\sqrt{(u^2-x^2)(4-u^2)}}\,.
\end{equation*}
The latter integral can be computed by means of the change of variables $u^2=4-(4-x^2)v$ and so
$$ p_{2n-3}(x)=2^{-n}C_n x^{2n-4}\,\int_0^1\left(1-\frac{4-x^2}{4}v\right)^{\frac{1-n}2}\frac{du}{\sqrt{v(1-v)}}\,. $$
Recall the Euler formula for the hypergeometric function \cite[Theorem 2.2.1]{AAR}
\begin{equation}\label{euler}
F(a,b;c;z)=\frac1{B(b,c-b)}\int_0^1\frac{v^{b-1}(1-v)^{c-b-1}}{(1-vz)^a}\,dv\,, \quad c>b>0,
\end{equation}
which gives
\begin{equation}\label{onedown}
p_{2n-3}(x)=C'_nx^{2n-4}\,F\Bigl(\frac{n-1}2,\frac12\,;1;\frac{4-x^2}{4}\Bigr).
\end{equation}

The first part of the algorithm is accomplished. We go now two steps forward by using 2-step transition formula \eqref{twoschoen2}
again with $m=2n-3$. Having in mind a formula for the derivative
\begin{equation}\label{deriv}
F'(a,b;c;z)=\frac{ab}{c}F(a+1,b+1;c+1;z)
\end{equation}
we put
$$ \varphi(x):=\frac{n-1}8\,x^{2n-2}F\Bigl(\frac{n+1}2,\frac32\,;2; \frac{4-x^2}4\Bigr). $$
It is clear that $\varphi(x)=O(1)$ as $x\to 2-0$, and (see, e.g., \cite[Theorem 2.1.3]{AAR})
\begin{equation}\label{asymp}
\varphi(x)=x^{2n-2}O(x^{-n})=O(x^{n-2}), \qquad x\to 0+0,
\end{equation}
so $\varphi$ is bounded and positive on $[0,2]$. By \eqref{deriv}
\begin{equation*}
\frac{\varphi(x)}{x^{2n-3}}=-\frac{d}{dx}\,F\Bigl(\frac{n-1}2,\frac12\,;1; \frac{4-x^2}4\Bigr), \quad
x^{2n-4}\int_x^2 \frac{\varphi(u)}{u^{2n-3}}\,du=C_n''p_{2n-3}(x)-x^{2n-4}.
\end{equation*}
Formula \eqref{twoschoen2} implies now
\begin{equation}
p_{2n-3}(x)=\frac{2x^{2n-4}}{B\bigl(\frac{2n-3}2,1\bigr)}\,\int_x^2 \frac{\nu_{2n-1}(du)}{u^{2n-3}}\,, \quad
\nu_{2n-1}(du)=A_n\varphi(x)dx+B_n\delta\{2\},
\end{equation}
$A_n,B_n$ are positive constants. So we see that $\Omega_n^2\in\Phi_{2n-1}$, and the corresponding Schoenberg measure $\nu_{2n-1}$
has a singular component. But Theorem \ref{prop_measure_relation} states that in this case $\Omega_n^2\notin\Phi_{2n}$, as claimed.

Finally, let $n=1$. Then
$$ \Omega_1^2(x)=\cos^2x=\frac{1+\cos 2x}2\,, \qquad \nu_1(\Omega_1^2)=\frac{\delta\{0\}+\delta\{2\}}2, $$
and again $\Omega_1^2\notin\Phi_2$. The proof is complete.
\end{proof}
\begin{remark}
We are unaware of the analog of formula \eqref{schoentran} for $\Omega_n^k$ with $k\ge3$.
Our conjecture with regard to the powers of $\Omega_n$ reads $\Omega_n^k\in\Phi_{kn-k+1}\backslash\Phi_{kn-k+2}$
(cf. Remark \ref{wideclass}).
\end{remark}
%%%%%%%%%%%%%%%%%%%%%%%%%%%%%%%%%%%%
%

By Proposition \ref{square}, each function

   \begin{equation}\label{2.22}
f(r) := \int^{\infty}_0\Omega^2_n(rt)\sigma(dt) %%
   \end{equation}
belongs to  the class $\Phi_{2n-2}$ whenever $\sigma$ is a probability measure on $\Bbb
R_+$. We put $f$ in the subclass $\Phi_{2n-2}^{(2)}\subset\Phi_{2n-2}$ if it admits
representation  \eqref{2.22}.

The following  result describes the class $\Phi_{2n-2}^{(2)}$ in terms of the
corresponding Schoenberg's measures.
   \begin{corollary}
Let $f\in \Phi_{2n-2}$ and let $\nu_{2n-2}$  be its Schoenberg's measure. Then $f\in
\Phi_{2n-2}^{(2)}$ if and only if
    \begin{equation}\label{2.23}
\nu_{2n-2}(du) = C_n\int^{\infty}_{u/2}\left(\frac{u}{t}\right)^{n-2}
\frac{\sigma(dt)}{\sqrt{4t^2-u^2}}\,du,
    \end{equation}
where $\sigma$ is a probability measure on $\R_+$. In particular, $\nu_{2n-2}$ is
absolutely continuous with respect to $\sigma$, and $\nu_{2n-2}\{(0,\varepsilon)\}>0$
for any $\varepsilon>0$.
  \end{corollary}
%%%%%%%%%%%%%%%%%%%%%%%%%%%%%%%%%%%%%%%%%%%%%%%%%%%%%%%%%%%%%
            \begin{proof}
Let $f\in \Phi_{2n-2}^{(2)}$. Then according to  \eqref{schoentran} %%by Proposition \ref{square},
    \begin{equation*}
    \begin{split}
f(r)&=\int^{\infty}_0 \sigma(dt) \int^{2}_0\Omega_{2n-2}(rtx)p_{2n-2}(x)\,dx
= C_n\int^2_0 \frac{x^{n-2}\,dx}{\sqrt{4-x^2}} \int^{\infty}_0\Omega_{2n-2}(rtx)\,\sigma(dt)  \\
& = C_n \int^{\infty}_0\,\sigma(dt)
\int^{2t}_0\Omega_{2n-2}(ru)\left(\frac{u}{t}\right)^{n-2}\frac{du}{\sqrt{4t^2-u^2}}  \\
& = C_n \int^{\infty}_0 \Omega_{2n-2}(ru)\,du
\int^{\infty}_{u/2}\left(\frac{u}{t}\right)^{n-2} \frac{\sigma(dt)}{\sqrt{4t^2-u^2}}\,.
\end{split}
    \end{equation*}
This proves representation \eqref{2.23}.

The converse statement is proved by reversing the reasoning.
            \end{proof}
%%%%%%%%%%%%%%%%%%%%%%%%%%%%%%%%%%%%

\section{Schoenberg matrices with infinitely many negative eigenvalues}
\label{secinf}

In view of the above results it seems reasonable to introduce the following
notations.
%%define the following values.

Given a finite set $Y\subset \R^n$, denote by $\kappa^-(g, Y)$ a number of negative
eigenvalues of the finite Schoenberg matrix $\kS_Y(g)$ counting multiplicity, and by

$$
\kappa_n^-(g):= \sup \{\kappa^-(g,Y): Y\  \text{runs through all finite subsets of}\  \R^n\}.
$$

Certainly, $\kappa_n^-(g)=0$ for $g\in\Phi_n$, and $\kappa_n$ is a nondecreasing function of $n$.

We turn to the case when the Schoenberg matrix $\kS_X(g)$ can have
arbitrarily many negative eigenvalues. As we will see shortly, the cases $n\ge2$ and $n=1$ should be discerned.

\subsection{The case $n\geq2$}

\begin{theorem}[=Theorem \ref{th1}]\label{negativeeigen2}
Let $g$ be a continuous function on $\R_+$ such that the limit
\begin{equation}\label{limit}
\lim_{t\to\infty} g(t)=g(\infty)\ge0
\end{equation}
exists. If $\kappa^-(g,Y)\ge1$ for some set $Y\in\R^m$, then $\kappa_m^-(g)=+\infty$. In particular,
$\kappa_{n+1}^-(g)=+\infty$ for each $g\in\Phi_n\backslash\Phi_{n+1}$ with $n\ge2$.
\end{theorem}
\begin{proof}
By the assumption there is a finite set $Y=\{y_{k}\}_{k=1}^p\subset\R^{m}$ so that $\kS_{Y}(g)$ has at
least one negative eigenvalue. Given an arbitrary positive integer $N$, consider
a collection of the shifts of $Y$ of the form $Y_j=Y+w_j$, $j=1,2,\ldots,N$ with
$$ w_j=(u_j,0,\ldots,0), \qquad 0=u_1<u_2<\ldots<u_N, $$
$u_j$ are chosen later on. Take $X=\cup_j Y_j$. The Schoenberg matrix $\kS_X(g)$ is now a block matrix with $p\times p$ blocks
\begin{equation}\label{infinitenegative}
\kS_X(g)=
\begin{bmatrix}
\kS_Y(g)&B_{12}&B_{13}&\ldots&B_{1N} \\
B_{21}&\kS_Y(g)&B_{23}&\ldots&B_{2N} \\
\vdots&\vdots& & &\vdots \\
B_{N1}&B_{N2}&B_{N3}&\ldots&\kS_Y(g)
\end{bmatrix}={\rm diag}(\kS_Y(g),\kS_Y(g),\ldots,\kS_Y(g))+\Delta.
\end{equation}
The block diagonal matrix in the right hand side of \eqref{infinitenegative} has at least $N$ negative eigenvalues.

Assume first that $g(\infty)=0$. Since the block entries of $\Delta$ are
$$ \Delta_{rs}=\|g(|y_i-y_j+w_r-w_s|)\|_{i,j=1}^p, \qquad r\not=s, $$
the appropriate choice of $\{u_j\}$ with large enough differences $|u_{j+1}-u_j|$ provides $\|\Delta_{rs}\|\le\varepsilon_1$. So
$$ \|\kS_X(g)-{\rm diag}(\kS_Y(g),\kS_Y(g),\ldots,\kS_Y(g))\|\le\varepsilon, $$
and $\kS_X(g)$ has at least $N$ negative eigenvalues, as claimed.

Let now $g(\infty)>0$. Then for $h=g-g(\infty)$ one has
$$ \kS_Y(h)=\kS_Y(g)-g(\infty)\,E_{p}\le \kS_Y(g). $$
So $\kS_{Y}(h)$ has at least one negative eigenvalue, and $\lim_{t\to\infty}h(t)=0$. By the above argument $\kappa_m^-(h)=+\infty$, so for any large $N$
there is a finite set $Z=\{z_j\}_{j=1}^l$ such that $\kS_Z(h)$ has at least $N$ negative eigenvalues. But
$$  \kS_Z(g)=\kS_Z(h)+g(\infty)\,E_{l}, $$
the matrix $E_{l}$ is nonnegative and has rank one, so $\kS_Z(g)$ has at least $N-1$ negative eigenvalues. Hence $\kappa_m^-(g)=+\infty$, as claimed.

To prove the second statement note that for $n\ge2$
$$ \Omega_n(x)=O(x^{\frac{1-n}2}), \qquad x\to\infty $$
in view of \eqref{kernel} and the well-known asymptotic behavior of the Bessel function $J_q(x)=O(x^{-1/2})$, $x\to\infty$.
As $\Omega_n(0)=1$, by the Dominated Convergence Theorem
$$ \lim_{x\to\infty}g(x)=\lim_{x\to\infty}\int_0^\infty \Omega_n(xt)\nu(dt)= \nu(\{0\})\ge0. $$
The result follows from the first statement.
\end{proof}

\begin{remark}
If $\kappa^-(g,Y)\ge2$ for some set $Y\in\R^m$, then $\kappa_m^-(g)=+\infty$ under assumption \eqref{limit} with an arbitrary value $g(\infty)$.
\end{remark}

%%%%%%%%%%%%%%%%%%%%%%%%%%%%%%%%%%%%%%%%%%%%%%%%%%%%%%%%

\subsection{The case $n=1$}

In the case $n=1$ the Schoenberg representation of $g\in \Phi_1$ reads (see \eqref{schoenberg1} and \eqref{3.13})
  \begin{equation}\label{schoenberg1New}
g(t) = \int_{0}^{+\infty}\cos(tu)\,\nu(du), \qquad t\in \R_+,
   \end{equation}
with a finite positive Borel measure $\nu=\nu_1$. Since $\Omega_1(s)=\cos s$ has no limit as $s\to\infty$, condition
\eqref{limit} is in general false, and Theorem \ref{negativeeigen2} does not apply directly. By using an ad hoc argument
we are able to prove $\kappa_2^-(g)=+\infty$ under certain additional assumptions on $\nu$.

Let us begin with a simple (and perhaps well-known) result about spectra of certain Toeplitz matrices.
\begin{lemma}\label{toeplitz}
Let $a=\{a_j\}_{j\in\Z}$ be an $m$-periodic sequence of complex numbers, $a_{j+m}=a_j$, $T_m(a):=\|a_{k-j}\|_{k,j=0}^{m-1}$. Then
the spectrum of $T_m(a)$ is
\begin{equation}
\sigma(T_m(a))=\{\lambda_{km}\}_{k=1}^m, \qquad \lambda_{km}=\sum_{j=1}^{m} a_{j}e^{-\frac{2\pi ik}{m}j}.
\end{equation}
If in addition $a_j$ are real numbers and $a_{-j}=a_j$ then
  \begin{equation} \label{4.10}
 \lambda_{km}=\sum_{j=1}^{m} a_{j}\cos\frac{2\pi k}{m}j\,.
   \end{equation}
\end{lemma}
\begin{proof}
It is easy to see that
$$ T_m(a)=a_0+a_{-1}Z+a_{-2}Z^2+\ldots+a_{-(m-1)}Z^{m-1}, $$
where $Z$ is the permutation matrix
\[
Z=
\begin{bmatrix}
0 & 1 & \ldots & 0 & 0& 0 \\
0 & 0 & 1 &\ldots& 0& 0 \\
0 & 0 & 0 & 1&\ldots & 0 \\
\ldots&\ldots&\ldots&\ldots& \ldots & \ldots \\
0 & 0 & 0 & \ldots& 0 & 1 \\
1 &\ldots & \ldots & \ldots & 0 & 0
\end{bmatrix}.
\]
The spectrum of $Z$ is well known, namely,
$$ \sigma(Z)=\{e^{\frac{2\pi i}{m}j}\}_{j=0}^{m-1}, $$
so by the Spectral Mapping Theorem the spectrum $\sigma(T_m(a))=\{\lambda_{km}\}_{k=1}^m$ is
 \begin{equation*}
\lambda_{km}=\sum_{l=0}^{m-1} a_{-l}e^{\frac{2\pi
ik}{m}l}=\sum_{l=0}^{m-1} a_{m-l}e^{\frac{2\pi
ik}{m}l}=\sum_{j=1}^{m} a_{j}e^{-\frac{2\pi ik}{m}j},
   \end{equation*}
as claimed. The second statement is obvious.
\end{proof}

It turns out that  the spectrum of the Schoenberg matrix $\kS_Y(g)$ can be computed explicitly when $Y$ is a set of vertices
of the regular $m$-gon on the complex plane.

  \begin{proposition}\label{simplex}
Let $Y_m(r)$ be the set of vertices of the regular $m$-gon of radius $r$, that is,
  \begin{equation}\label{Poligon2}
Y_m(r)= \{y_k(r)\}_{k=1}^m, \qquad y_k= y_k(r) := re^{\frac{2\pi ik}{m}}, \quad k=1,2,\ldots,m.
 \end{equation}
Given an even continuous function $g$ on $\R$, let $\kS_{Y_m(r)}(g)=\|g(|y_j-y_k|)\|_{j,k=1}^m$ be the corresponding Schoenberg's matrix.
Then for its spectrum one has
\begin{equation}\label{polygon3}
\sigma\bigl(\kS_{Y_m(r)}(g)\bigr) = \{\lambda_{km}(r)\}_{k=1}^m, \qquad
\lambda_{km}(r) = \,\sum_{j=1}^{m} g\biggl(2r\sin\frac{\pi j}m\biggr)\,\cos\frac{2\pi k}{m}j\,.
\end{equation}
In particular,
\begin{equation}\label{specttoep}
\lim_{m\to\infty} \frac{\lambda_{km}(r)}m = \widehat
g(k,r):=\frac1{2\pi}\int_0^{2\pi} g\biggl(2r\sin\frac{t}2\biggr)\cos
kt\,dt.
   \end{equation}
\end{proposition}
\begin{proof}
It is clear that the Schoenberg matrix $\kS_{Y_m(r)}(g)$ is of the type considered in Lemma \ref{toeplitz}
$$
\kS_{Y_m(r)}(g)  = \|g(|y_j(r) - y_k(r)|)\|_{j,k=1}^m =
\left\|g\biggl(2r\sin\frac{\pi |k-j|}m \biggr)\right\|_{j,k=1}^m.
$$
Hence the result is immediate from \eqref{4.10}. Equality \eqref{polygon3} divided by $m$ gives an integral sum for the integral in
\eqref{specttoep}. The proof is complete.
\end{proof}

Given a sequence of (not necessarily different) numbers $\alpha = \{\alpha_k\}_{k=1}^l\subset \R$, $l\in \N\cup \{\infty\}$,
we denote by $\kappa_-(\alpha)$ a number of negative entries in  $\alpha$ counting multiplicity. Similarly, given a
symmetric matrix $A$ we denote by $\kappa_-(A)$ a number of its negative eigenvalues counting multiplicity.

%%%%%%%%%%%%%%%%%%%%%%%%%%%%%%%%%%%%%%%%%%%%%%%%%%%%%%%%%%%%%%%%
  \begin{corollary}\label{cor4.6}
Let $\kK(r):=\{\widehat g(k,r)\}_{k\in\N}$. If $\,\sup_{r>0} \kappa_-(\kK(r)) = +\infty$,  then $\kappa_2^-(g) = +\infty$. In
particular, if $\kappa_-(\kK(r_0)) = +\infty$ for some $r_0 > 0$, then $\kappa_2^-(g) = +\infty$.
\end{corollary}
%%%%%%%%%%%%%%%%%%%%%%%%%%%%%%%%%%%%%%%%%%%%%%%%
  \begin{proof}
By the assumption, for an arbitrary $N\in\N$ there is $r>0$ so that $\kK(r)$ contains at least $N$ negative
numbers $\{\widehat g(k_j,r)\}_{j=1}^N$. By Proposition \ref{simplex} for large enough $m$
$$ \lambda_{k_j,m}<0, \qquad j=1,2,\ldots,N, $$
so there are at least $N$ negative eigenvalues counting multiplicity of the Schoenberg matrix $\kS_{Y_m(r)}(g)=\|g(|y_j-y_k|)\|_{j,k=1}^m$.
Hence  $\kappa_2^-(g) = +\infty$, as claimed.
     \end{proof}
Clearly, if the measure $\nu$ in \eqref{schoenberg1New} is absolutely continuous, $\nu(du)=p(u)du$,
then by the Riemann--Lebesgue Lemma $g(t)\to 0$ as $t\to\infty$. So, Theorem \ref{negativeeigen2} applies and $\kappa_2^-(g)=+\infty$
as long as $g\notin\Phi_2$. In what follows we will focus upon the opposite case when $\nu$ contains a singular component.

We begin with the simplest case $\nu=\delta\{s\}$.
%%%%%%%%%%%%%%%%%%%%%%%%%%%%%%%%%%%%%%%%
  \begin{proposition}
Let $g_s(r)=\cos sr$, $s>0$. Then $\kappa_2^-(g_s) =+\infty$.
  \end{proposition}
%%%%%%%%%%%%%%%%%%%%%%%%%%%%%%%%%%%%%%%%%%%%%
  \begin{proof}
To apply Corollary  \ref{cor4.6}  we compute  the cosine Fourier coefficients
of the function $g_s\bigl(2r\sin\frac{t}2\bigr)$. Fortunately, this can be done explicitly. Precisely,
we have (see \cite[p. 21]{Wat})
  \begin{equation}\label{cosbes}
\widehat g_s(k,r) := \frac1{2\pi}\int_0^{2\pi}
\cos\biggl(2sr\sin\frac{t}2\biggr)\cos kt\,dt =
 \frac1{\pi}\int_0^{\pi} \cos(2sr\sin\theta)\cos 2k\theta\,d\theta = J_{2k}(2sr).
  \end{equation}
According to Corollary \ref{cor4.6} the problem is reduced to the study of a number of negative terms in the sequence
$\{J_{2k}(x)\}_{k=1}^\infty$ for $x>0$ large enough. Our argument relies heavily on the well-known identity \cite[p. 17]{Wat}
\begin{equation}\label{recurbes}
\frac{J_{p-1}(x)+J_{p+1}(x)}{p}=\frac{2J_p(x)}{x}.
\end{equation}
Write \eqref{recurbes} with $p=2k-1$, $p=2k+1$
\begin{equation*}
\begin{split}
\frac{J_{2k-2}(x)}{2k-1} + \frac{J_{2k}(x)}{2k-1} &=\frac{2J_{2k-1}(x)}{x}, \\
\frac{J_{2k}(x)}{2k+1} + \frac{J_{2k+2}(x)}{2k+1} &=\frac{2J_{2k+1}(x)}{x},
\end{split}
\end{equation*}
take their sum, and apply again \eqref{recurbes} with $p=2k$ to obtain
\begin{equation}\label{recurbes1}
\begin{split}
\frac{J_{2k-2}(x)}{2k-1}+\frac{4k}{4k^2-1}\,J_{2k}(x)+\frac{J_{2k+2}(x)}{2k+1} &=\frac2{x}\bigl(J_{2k-1}(x)+J_{2k+1}(x)\bigr), \\
\frac{J_{2k-2}(x)}{2k-1}+4k\biggl(\frac{1}{4k^2-1}-\frac{2}{x^2}\biggr)\,J_{2k}(x)+\frac{J_{2k+2}(x)}{2k+1} &\equiv 0.
\end{split}
\end{equation}

Denote by $Z$ a set of all positive roots of at least one function $J_{2k}$, $k\in\N$. We assume later on that $x\notin Z$, so $J_{2k}(x)\not=0$
for all $k$ (note that $Z$ is a countable subset of $\R_+$).

Given $N\in\N$ put $x=9N+\varepsilon_N$, $0<\varepsilon_N<1$, so that $x\notin Z$. It is clear that
$$ \frac{1}{4k^2-1}-\frac{2}{x^2}>0, \qquad k=1,2,\ldots,3N, $$
so by \eqref{recurbes1} at least one of the numbers $J_{2k-2}(x)$, $J_{2k}(x)$, $J_{2k+2}(x)$ is negative. There are exactly
$N$ such triples in the set $\{J_{2p}(x)\}_{p=1}^{3N}$, so
\begin{equation}
\kappa_-\bigl(\{J_{2p}(x)\}_{p=1}^{3N}\bigr)\ge N.
\end{equation}
Hence, by Corollary \ref{cor4.6}, $\kappa_2^-(g_s)=+\infty$, as claimed.
\end{proof}

\begin{remark}\label{omegasquare}
It is easy to see that the same conclusion holds for the function
$$ g_s^2(r)=\cos^2 sr=\frac{1+\cos 2sr}2\,. $$
Indeed, the Schoenberg matrix $\kS_X(g_s^2)$ is the rank one perturbation of the Schoenberg matrix $\kS_X(g_s)$, and the latter
can have arbitrarily many negative eigenvalues for an appropriate choice of the set $X\subset\R^{2}$.
\end{remark}

We show that the same conclusion remains valid within a certain class of Schoenberg's measures.

\begin{proposition}
Suppose that for a function $g$ $\eqref{schoenberg1New}$ the support of $\nu$ is separated from the origin, i.e.,
$\supp\nu\subset [a,\infty)$ for some $a>0$. Next, assume that
\begin{equation}\label{nontozero}
l(\nu):=\limsup_{r\to\infty} |h(r)|>0, \qquad h(r):=\frac1{2\pi^{3/2}}\,\int_0^\infty \frac{\cos\left(2rs-\frac{\pi}4\right)}{\sqrt{s}}\,\nu(ds).
\end{equation}
Then $\kappa_2^-(g)=+\infty$.
\end{proposition}
\begin{proof}
It follows from \eqref{schoenberg1New} and \eqref{cosbes} that
\begin{equation}
\widehat g(k,r)=\frac1{2\pi}\int_0^\infty J_{2k}(2rs)\nu(ds).
\end{equation}
We wish to show that the number of negative terms in the sequence $\{\widehat g(k,r)\}_{k\in\N}$ grows unboundedly as $r\to\infty$.

Write the asymptotic expansion \cite[formula (4.8.5)]{AAR}
\begin{equation*}
J_{2k}(x)=\sqrt{\frac2{\pi x}}\cos\left(x-k\pi-\frac{\pi}4\right)+\varepsilon_{2k}(x)=
(-1)^k\sqrt{\frac2{\pi x}}\cos\left(x-\frac{\pi}4\right)+\varepsilon_{2k}(x),
\end{equation*}
and so
\begin{equation}
\widehat g(k,r)=(-1)^k\frac{h(r)}{\sqrt{r}}+\frac1{2\pi}\int_0^\infty \varepsilon_{2k}(2rs)\nu(ds).
\end{equation}
The following uniform bound for the reminder $\varepsilon_{2k}$ is known \cite[Theorem 10]{kras}
$$ \varepsilon_{2k}(x)\le c_1\frac{k^2}{x^{3/2}}, \qquad x>0, $$
where $c_p$ below stand for some absolute constants . Hence
\begin{equation}\label{gsign}
|\sqrt{r}\widehat g(k,r)-(-1)^kh(r)|\le c_2\,\frac{k^2}{r}\int_a^\infty\frac{\nu(ds)}{s^{3/2}}\,.
\end{equation}
By assumption \eqref{nontozero} there is a sequence of positive numbers $\{r_j\}_{j\ge1}$ so that
\begin{equation*}
\lim_{j\to\infty}r_j=+\infty, \qquad |h(r_j)|\ge \frac{l(\nu)}2>0.
\end{equation*}
Take large enough $j_0=j_0(\nu)$ so that
$$ c_2\,\int_a^\infty\frac{\nu(ds)}{s^{3/2}}<\frac{l(\nu)}4\,r_j^{1/2}, \quad j\ge j_0. $$
It follows now from \eqref{gsign} that
$$ |\sqrt{r}\widehat g(k,r)-(-1)^kh(r)|<\frac{l(\nu)}4, \qquad k=1,2,\ldots,[r_j^{1/4}], \quad j\ge j_0, $$
and the number of negative terms in $\{\widehat g(k,r)\}_{k\in\N}$ is at least $c_3r_j^{1/4}$, so it grows to infinity as $j\to\infty$.
An application of Corollary \ref{cor4.6} completes the proof.
\end{proof}

\begin{remark}
As a matter of fact, the assumption $\supp\nu\subset [a,\infty)$ can be relaxed to
\begin{equation}
\int_0^\infty \frac{\nu(ds)}{\sqrt{s}}<\infty, \qquad \int_x^\infty \frac{\nu(ds)}{s^{3/2}}=o(x^{-1}), \quad x\to 0.
\end{equation}
\end{remark}

\begin{example}
Let
$$ \nu(ds)=\sum_{k=1}^\infty a_k\delta\{s_k\}, \qquad \inf_k s_k>0, $$
with $a_k\ge0$, $\{a_k\}_{k\ge1}\in\ell^1(\N)$.
Then \eqref{nontozero} is true since $h$ is almost periodic on $\R_+$.
\end{example}

\begin{remark}
The similar problem for a general Bochner's class $P_n$ of positive-definite functions on $\R^n$ can be easily resolved.
Recall that a continuous function $f$ on $\R^n$ belongs to $P_n$ if for an arbitrary finite set $\{x_1,\dots,x_m\}$, $x_k\in \R^n$,
and $\{\xi_1,\dots,\xi_m\}\in\C^m$
\begin{equation*}
\sum_{k,j=1}^{m} f(x_k-x_j)\xi_j\overline{\xi}_k\ge 0.
\end{equation*}
Each function from the Bochner class admits the following integral representation
\begin{equation}\label{bochner}
f(x)=\int_{\R^n} e^{i(x,t)_n}\, \sigma(dt),
\end{equation}
where $\sigma$ is a finite positive Borel measure on $\R^n$.

Although inclusion of the classes $P_n$ has no sense now, let us write $f_{n+1}\in P_{n-1}\backslash P_{n}$ for a function $f_{n}$ on $\R^{n}$ if
$f_{n}\notin P_{n}$ but its restriction on $\R^{n-1}$
$$ f_{n-1}(x_1,x_2,\ldots,x_{n-1}):=f_{n}(x_1,x_2,\ldots,x_{n-1},0)\in P_{n-1}. $$

Let $\mu$ be a finite positive Borel measure on $\R^{n}$ with the only atom at the origin $\mu\{0\}>0$ and the Fourier transform
$f_{n}$ \eqref{bochner}. Put
\begin{equation}\label{charge}
\mu_1=\mu-\varepsilon\,\delta_a, \qquad 0<\varepsilon<\mu\{0\}, \quad a=(0,\ldots,0,1)\in\R^{n},
\end{equation}
a finite charge (sign measure) on $\R^{n}$. By the construction, $\mu$ is not a measure, so the Fourier transform
\begin{equation}\label{ftcharge}
g_{n}(x)=\int_{\R^{n}} e^{i(x,t)_{n}}\,\mu_1(dt)=f_{n}(x)-\varepsilon\,e^{ix_{n}}\notin P_{n}.
\end{equation}
So
$$ g_{n-1}(x)=f_{n-1}(x)-\varepsilon=\int_{\R^{n-1}} e^{i(x,t)_{n-1}}\,\tilde\mu(dt)-\varepsilon, $$
where $\tilde\mu$ is a projection of $\mu$ on $\R^{n-1}$, that is, $\tilde\mu(E)=\mu(E\times\R)$ for each Borel set $E\subset\R^{n-1}$.
By the choice of $\varepsilon$ \eqref{charge} and $\tilde\mu\{0\}\ge\mu\{0\}$ we have $g_n\in P_{n-1}$, so $g_{n}\in P_{n-1}\backslash P_{n}$.

Given a finite set $Y=\{y_j\}_{j=1}^N\subset\R^{n}$ and a function $g$ on $\R^{n}$, we define a finite Schoenberg matrix by
$\kS_Y(g)=\|f(y_i-y_j)\|_{i,j=1}^{N}$, and define the values $\kappa^-(g,Y)$ and $\kappa_{n}^-(g)$ as above in Section~1.
It follows from \eqref{ftcharge} that $\kS_Y(g_{n})=T_1-T_2$, where $T_1\ge0$ and ${\rm rk}\,T_2=1$. So $\kappa(g_{n},Y)\ge 1$, and for some
particular choice of $Y$ $\kappa(g_{n},Y)=1$. Hence, $\kappa_{n}^-(g_{n})=1$, as needed.
     \end{remark}
%%%%%%%%%%%%%%%%%%%%%%%%%%%%%%%%%%%%%%%%%%%%%%%%%%%%%%%%%%%%%%%%%%%%%%%%%%%%%%%%%%%%%%%%%%
%
%

   %%%%%%%%%%%%

 \bigskip
 \bigskip

Leonid Golinskii, \\
\emph{Mathematics Division, Low Temperature Physics Institute, NAS of Ukraine},\\
\emph{47 Lenin ave.},\\
\emph{61103 Kharkov, Ukraine}\\
\emph{e-mail:} golinskii@ilt.kharkov.ua \\

Mark Malamud,\\
\emph{Institute of Applied Mathematics and Mechanics, NAS of Ukraine},\\
\emph{74 R. Luxemburg str.},\\
\emph{83114 Donetsk, Ukraine}\\
 \emph{e-mail:} mmm@telenet.dn.ua
 \\

Leonid Oridoroga, \\
\emph{Donetsk National University}, \\
\emph{24, Universitetskaya Str.}, \\
\emph{83055  Donetsk, Ukraine}  \\
\emph{e-mail:} oridoroga@skif.net
\\

\end{document}